\documentclass[10pt,conference, onecolumn, letterpaper,compsoc]{IEEEtran}

\usepackage[centertags]{amsmath}
\usepackage[english]{babel}
\usepackage{amsfonts,amsthm}
\usepackage{graphicx}
\usepackage{ifthen}
\usepackage{algorithm}
\usepackage{algorithmic}
\usepackage{multirow}
\usepackage{changepage}
\usepackage{caption}
\usepackage{pgfplots}
\pgfplotsset{compat=newest}
\def\({\left(}
\def\){\right)}

\newtheorem{problem}{Problem}
\newcommand{\blambda}{\mbox{\boldmath$\lambda$\unboldmath}}

\newcommand{\R}{\mathbb{R}}
\newcommand{\Q}{\mathbb{Q}}
\newcommand{\Z}{\mathbb{Z}}
\newcommand{\V}{\textbf{V}}

\newcommand{\modulo}[2]{\left\langle #1\right\rangle_{#2}}

\def\vol{\mathrm{Vol}}

\newcommand{\cL}{\mathcal{L}}
\renewcommand{\Z}{\mathbb{Z}}
\renewcommand{\vec}[1]{\textbf{#1}}
\newcommand{\mat}[1]{\textbf{#1}}

\newtheorem{propo}{Proposition}

\newtheorem{Lm}{Lemma}

\usepackage{url}

\begin{document}
\title{The MMO Problem}
\author{\IEEEauthorblockN{Oscar Garc\'{\i}a-Morch\'on\\
Ronald Rietman\\
 Ludo Tolhuizen}
\IEEEauthorblockA{Philips Research\\ Eindhoven, The Netherlands}
\and
\IEEEauthorblockN{Domingo G\'omez\\
Jaime Guti\'errez}
\IEEEauthorblockA{Universidad de Cantabria\\ Santander, Spain}
}

\maketitle

\begin{abstract}
We consider a two polynomials  analogue of the polynomial interpolation problem. 
Namely, we consider the Mixing Modular Operations (MMO)  problem of recovering
two polynomials $f\in \Z_p[x]$ and $g\in \Z_q[x]$ of known  degree, where $p$
and $q$ are two (un)known positive integers, from the values of $f(t)\bmod p +
g(t)\bmod q$  at polynomially many points $t \in \Z$.  
We  show that if $p$ and $q$ are known, the MMO problem is equivalent to
computing  a close vector in a lattice with respect to  the infinity norm.  We
also  implemented in the SAGE system a heuristic polynomial-time  algorithm.  
If $p$ and $q$ are kept secret, we do not know how to solve this problem.  This
problem is motivated by several potential cryptographic applications.
 \end{abstract}

\section{Introduction}

For integer $x$ and integer $p\geq 2$, we denote by $\langle x\rangle_p$ the remainder of dividing $x$ by $p$. 
Stated differently, 
\[ 0 \leq \langle x\rangle_p \leq p-1 \mbox{ and } x \equiv \langle x \rangle_p \bmod p .  \]
The set $\{0,1,\ldots,p-1\}$ can be identified with $\Z_p$, the ring of integers modulo $p$. 
Conversely, $\Z_p$ can be considered as a subset of $\Z$. 
This allows us to interpret functions on $\Z_p$ as polynomials evaluated modulo~$p$ on the set $\{0,1,\ldots,p-1\}$ and to extend the domain of these polynomials to $\Z$. Furthermore it allows us to add polynomials over several different rings $\Z_{p}, \Z_{q},\ldots$ for different values of the moduli $p,q,\ldots$.
This addition we denote by the term Mixing of Modular Operations.

Here we study a variant of the very well known polynomial interpolation problem,  where the function to be interpolated is the sum of two polynomials
reduced modulo two different unknown numbers $p$ and $q$.

\begin{problem}
  \label{prob:MMO_basic_secret}
 Let  $p \not= q$ be two  positive unknown integers and $c$ another positive integer.  
Let the function $h:\Z\to\Z$ be the sum of two unknown reduced polynomials
$h(x)=\langle f(x)\rangle_p + \langle g(x)\rangle_q$ for some polynomials
$f\in\Z_p[x]$, $g\in\Z_q[x]$ of degree at most $\alpha$, where $\alpha$ is
known. Suppose that the set  
  \begin{equation*}
    J = \{(x_1,h(x_1)),\ldots, (x_c,h(x_c))\}
  \end{equation*}
is  known, where $x_i \in \Z$, for $i=1,\ldots,c$.  
The \textbf{MMO problem} is to recover $p$ and $q$ and
the polynomials $f$ and $g$.
\end{problem}

This problem seems to be difficult to solve even for very small polynomial degrees and, in fact, we and other colleagues have not managed.
For the single-polynomial analogue of Problem~\ref{prob:MMO_basic_secret}, we refer to
related work in \cite{Boyar}.
The main motivation to study this computational problem arises in 
potential applications  to cryptography \cite{MorchonTGG12}.
Since we could not obtain a solution for the above problem, this paper mainly devotes its attention to a simplified problem statement in which $p$ and $q$ are known.

\begin{problem}
  \label{prob:MMO_basic}
 Let  $p \not= q$ be two  known
positive integers and $c$ another positive integer.  
Let the function $h:\Z\to\Z$ be the sum of two unknown reduced polynomials
$h(x)=\langle f(x)\rangle_p + \langle g(x)\rangle_q$ for some polynomials
$f\in\Z_p[x]$, $g\in\Z_q[x]$ of degree at most $\alpha$. Suppose that the set  
  \begin{equation*}
    J = \{(x_1,h(x_1)),\ldots, (x_c,h(x_c))\}
  \end{equation*}
is  known, where $x_i \in \Z$, for $i=1,\ldots,c$.  
The  \textbf{MMO problem with known moduli} is to recover the polynomials $f$
and $g$.
\end{problem}

This is a natural extension of the well known  polynomial interpolation problem. 
Our results show that if $c$ is big enough compared to $\alpha$, and the points $x_1,\ldots,x_c$ are randomly drawn from a large enough interval, the MMO problem has a unique solution $f$, $g$, up to an additive constant.

The paper is organized as follows: 
Section 2 gives the equivalence of  the MMO problem to finding all points in a
lattice
 of dimension 
$c+2\alpha$ that are close to a target vector with respect to the
infinity norm.   
Section 3 shows the performance of a Sage implementation of the
provided  heuristic algorithm.  
In Section 4, we consider the MMO problem for the case that all arguments $x_i$ lie in a short interval.
Section 5 concludes this paper.
  
\section{ A general approach}

\subsection{Preliminaries}
This section is devoted to the preliminaries needed to understand the
results in the paper. Our purpose is not to give a deep treatment of
lattices because these are used in this article only as technical tools.
For a nice overview from a cryptographic perspective, we recommend the
reader~\cite{micciancio-goldwasser}. If the reader interests are
nearer to the area of number theory, we recommend~\cite{Gruber}.

Let $\{{\vec{a}}_1,\ldots,{\vec{a}}_d\}$ be a set of linearly
independent row vectors in ${\R}^s$. The set
$$
{\cal L}=\{\vec{z}  \ : \ \vec{z}=c_1\vec{a}_1+\ldots+
c_d\vec{a}_d,\quad c_1, \ldots, c_d\in\Z\}
$$
is called an {\it $d$-dimensional  lattice\/} with  {\it basis\/}
$\{ {\vec{a}}_1,\ldots, {\vec{a}}_d\}$.

To each lattice ${\cal L}$ one can naturally associate its {\it
volume\/}
$$
\vol({\cal L}) = \(\det \(B B^t\)\)^{1/2},
$$
where $B\in \mathbb{R}^{d\times s}$ is the matrix with rows $\vec{a}_1,\ldots,
\vec{a}_d$. The lattice volume is invariant under unimodular transformations
of  the   basis $\{{\vec{a}}_1,\ldots,{\vec{a}}_d\}$.

For a vector $\vec{u}$, let $\|\vec{u}\|_\infty$ denote its  {\it
infinity norm\/} and  by  $\|\vec{u}\|_2$  its  {\it
Euclidean norm\/}.   It is well known that:
$$\|\vec{u}\|_\infty \leq  \|\vec{u}\|_2 \leq  \sqrt s \|\vec{u}\|_\infty .$$
Any basis of a lattice satisfies 
$$\vol({\cL})\le \prod_{i=1}^{d}\|\vec{a}_i\|_2.$$

The famous Minkowski theorem  (see~\cite[Theorem 5.3.6, page
141]{GrLoSch}) gives an upper bound on $s_\infty(\cL)$, the length in
infinity-norm of
 the shortest nonzero vector in any $d$-dimensional lattice
${\cal L}$, in terms of its volume:
\begin{equation}
\label{MinkB} s_\infty(\cL) =
\min \left\{ \|\vec{z}\|_\infty  \colon \ \vec{z}  \in
{\cL} \setminus \{\vec{0}\}\right\} \le  \vol({\cL})^{1/d}
\end{equation}

Denote the number of points of a $d$-dimensional lattice in $\R^d$ that lie in
a measurable subset $S$ of $\R^d$ by $N_\cL(S)$. Let $C(\cL)$ be a fundamental
cell of $\cL$, with volume $\vol(\cL)$, The mean number of
lattice points in the shifted set $\vec{x}+S$, where $\vec{x}\in C(\cL)$,
is given by
\[ \frac{1}{\vol(\cL)}\int_{C(\cL)} N_\cL(\vec{x}+S)\,d^dx =
\frac{\vol(S)}{\vol(\cL)}.\]A similar result appears in~\cite[Lemma 2, page 27]{BookLLL}, where
the number of lattice points inside a $d$-dimensional ball of radius $r$ is
approximated by the volume of the ball divided by the volume of the
lattice.

As in~\cite[Definition 8, page 27]{BookLLL}, the \emph{Gaussian
  heuristic}  is to neglect the averaging, and estimate the
number of lattice points in $S$ as
\[ N_\cL(S) \approx \frac{\vol(S)}{\vol(\cL)}. \]
Take $S$ to be a $d$-dimensional hypercube  of length $2L$, parallel to the
coordinate axes and
centered around a lattice point. For
$S$ to contain one lattice point, $L$ must be less than $s_\infty(\cL)$.
The Gaussian heuristic thus suggests that $(2 s_\infty(\cL) )^d > \vol(\cL)$,
giving a lower bound
\[ s_\infty(\cL) > \frac{1}{2} \( \vol(\cL) \)^{1/d},\]
which is precisely half as big as the rigorous upper bound given by the
Minkowski theorem.

Finding the shortest vector in the lattice is a difficult
task. Indeed, finding the shortest vector of a lattice for the
infinity norm is $NP-$hard. Fortunately, after the breakthrough
in~\cite{LLL82}, it is possible to find ``short'' vectors in a
lattice, thanks to the concept of \textit{$LLL$-reduced basis}.
For the $LLL-$reduced basis $\vec{a}_1,\ldots,\vec{a}_d$ and its
Gram-Schmidt orthogonalization $\vec{a}^*_1,\ldots,\vec{a}^*_d$ there exist
real numbers $\mu_{ij}$ for $1\leq j \leq i\leq d$ such that
\begin{eqnarray*}
  &|\mu_{ij}|\le 1/2, \text{ for } 1\le j<i\le d,\\
  &\|\vec{a}_i^*+\mu_{ii-1}\vec{a}_{i-1}^*\|_2^2\le \epsilon
  \|\vec{a}_{i-1}^*\|_2^2,   \text{ for } i= 1\ldots, d-1. 
\end{eqnarray*}
for some $\epsilon \in (1/4, 1)$. 

  Finally, we introduce the following notation. For each real $x$, we denote by $\lfloor x \rfloor$  the value of $x$ rounded downwards to the
closest integer, that is,
\[ \lfloor x \rfloor = \max \{ m\in\Z \mid m\leq x\}.
\]

\begin{Lm}
\label{bound_infinity} For any integer $x$ and any integer $p>1$, we have:
\begin{itemize}
\item  $\modulo{x}{p}= x- \lfloor x/p\rfloor p$
\item There is a unique  integer 
$\lambda$ such that $| 2x - 2p\lambda - (p-1)| < p $. For this integer it holds
that $\lambda=  \lfloor x/p\rfloor$.
\end{itemize}
\end{Lm}

Similarly, for an integer vector $\vec{x}=(x_1,\ldots,x_d)$, $\lfloor
\vec{x}/p \rfloor$ is equal to the unique integer vector $\blambda=(\lambda_1,\ldots, \lambda_d)$
such that for each component it holds that 
$| 2x_k - 2p\lambda_k - (p-1)| < p$. 
If $\vec{e}_{d}$ is the vector of length
$d$ with all components equal to $1$, the latter 
condition is equivalent to $\|2\vec{x} - 2p\blambda -
(p-1)\vec{e}_d\|_{\infty} < p$.

\subsection{Lattice reduction}

The next proposition shows that from the values of $h(x)=\langle f(x) \rangle_p + \langle g(x)\rangle_q$ in all
integers $x$, the polynomials $f\in\Z_p[x]$ and $g\in\Z_q[x]$ are determined uniquely up to constant.
\begin{propo}
Let $p$ and $q$ be two positive integers that are relatively prime.
Let $f,g,u,v$ be functions from $\Z$ to $\Z$ such that for each integer $x$, 
\[ \langle f(x) \rangle_p + \langle g(x) \rangle_q = \langle u(x) \rangle_p + \langle v(x)\rangle_q. \]
There exists an integer $C$ such that for each integer
$x$, we have that
\[ \langle u(x)\rangle_p = \langle f(x) \rangle_p +C \mbox{  and }
  \langle v(x) \rangle_q  = \langle g(x) \rangle_q -C . \]
\label{propo:unique}
\end{propo}
\begin{proof}
For each integer $x$, we have that 
\[ \langle f(x)\rangle_p - \langle u(x) \rangle_p = \langle v(x) \rangle_q -
\langle g(x) \rangle_q . \]
The function $\langle f(x) \rangle_p - \langle u(x) \rangle_p$, which
clearly is periodic with period $p$,
thus also is periodic with period $q$, and thus is periodic with
period gcd($p,q$)=1, that is, the function is constant.
\end{proof}
Since it must hold, for every $x$, that 
$0\le \langle f(x)\rangle_p, \langle
u(x)\rangle_p \le p-1$ and
$0\le \langle g(x)\rangle_q, \langle v(x)\rangle_q \le q-1$,
it follows that, for all $x$
\[ \max( -\langle f(x)\rangle_p, \langle g(x)\rangle_q - q +1) \leq
C\leq \min(p-1-\langle f(x)\rangle_p, \langle g(x)\rangle_q). \]
In particular, $C$ must be equal to $0$, and thus the decomposition of the
function $h$ must be unique, if there is an $x$ for which
$\langle f(x)\rangle_p=\langle g(x)\rangle_q=0$. 
That is the reason why we will suppose that $f(0)=g(0)=0$.  
Additionally, we are going to suppose $\gcd(p,q)=1$.  Under this
condition there exist integers $\mu_1$ and $\mu_2$ such that
$\mu_1 p + \mu_2 q = 1$. We want to mention now that if $p$ is 
much larger than $q$, then the MMO problem with known moduli can be easily transformed
in a \textit{noisy polynomial interpolation problem} (see~\cite{Igor-Arne}), 
where the evaluation of the polynomial $g$ modulo $q$
can be seen as random ``noise'' and the attacker tries to recover $f$. 
Rigorous bounds for the noise of the results in~\cite{Igor-Arne} depend heavily 
on the performance of finding a close vector in the lattice and it seems that 
there is some gap between the theoretic results and the practical experiments. 
For this paper, we focus on the case that $p$ and $q$ have approximately the same number of bits.

Without loss of generality, the expression of the polynomials $f,g$ is   
\begin{equation*}
  f(x) = \sum_{k=1}^{\alpha}r_{k}x^{k},
  \quad 
  g(x) = \sum_{k=1}^{\alpha}t_{k}x^{k},
\end{equation*}
where $r_k, t_k\in\Z$ and
$|r_k|<  p/2$, $|t_k|< q/2$ for $k=1,\ldots, \alpha$.
 
We will show that the  MMO problem is related to finding a short vector  in a
lattice. For that, we need the following definitions:

From  $x_1,\ldots,x_c$  we
build the Vandermonde matrix $\V$ of size $\alpha\times c$ as
\begin{equation*}
\V=
\begin{pmatrix}
 x_1 & x_2 & \cdots & x_c \\
 x_1^2 & x_2^2 & \cdots & x_c^2 \\
 \vdots & \vdots & \ddots & \vdots \\
 x_1^\alpha & x_2^\alpha & \cdots & x_c^\alpha 
\end{pmatrix}.
\end{equation*}
Also, for integer $x$ we denote by  $h(x) = 
\modulo{f(x)}{p}+\modulo{g(x)}{q}$. 
  The MMO problem can now be formulated as follows: 
given the vector $\vec{h}$ of which the components are the function
values  $\vec{h}= (h(x_1),\ldots, h(x_c))$, find
integer vectors $\vec{r},\vec{t}$ of length $\alpha$
such that $\|\vec{r}\|_{\infty}< p/2$,
$\|\vec{t}\|_{\infty}< q/2$ and 
\begin{equation*}
  \vec{h} = \modulo{\vec{r}\V}{p} + \modulo{\vec{t}\V}{q} 
  = \vec{r}\V -  p\lfloor \vec{r}\V/p \rfloor + \vec{t}\V - 
  q\lfloor \vec{t}\V/q \rfloor 
\end{equation*}
where all the  modulo and rounding operations act
component-wise. 

Using Lemma \ref{bound_infinity}, it is clear that the MMO problem can be
restated as follows:
given $\vec{h}$, find integer row vectors $\vec{r},\vec{t}$ of
length $\alpha$ and $\blambda_1,\blambda_2$ of length $c$ such that
\begin{equation}
\label{eq:lateq}
\vec{h} = \vec{r}\V-p \blambda_1 + 
\vec{t}\V-q \blambda_2,
\end{equation}
and 
\begin{equation}
\label{eq:latnorm1}
 \left|\left| \frac{\vec{r}\V}{p} -  \blambda_1 
- \frac{(p-1)\mathbf{e}_c}{2p} \right|\right|_\infty  < \frac{1}{2}
,\ 
 \left|\left| \frac{\vec{t}\V}{q} -  \blambda_2 
- \frac{(q-1)\mathbf{e}_c}{2q} \right|\right|_\infty
< \frac{1}{2}.
\end{equation}
The inequalities in~\eqref{eq:latnorm1} embody the constraints that
the vectors $\blambda_1,\blambda_2$ are 
the result of the rounding operation.

We concatenate the vectors $\vec{r}$, $\vec{t}$, $\blambda_1$ and
$\blambda_2$
vector $\vec{x}$ of length $2(c+\alpha)$:
$$ \vec{x} = (\vec{r},\vec{t},-\blambda_1,-\blambda_2)$$
and define a matrix $\vec{A}$ of size 
$2(c+\alpha) \times c$ as a vertical
concatenation of 
$2$ copies of $\vec{V}$ and
$2$ instances of the
$c\times c$ identity matrix $\mat{I}_c$ multiplied by $p,$ $q$
respectively:  
$$ \mat{A} = \begin{pmatrix}
 \V \\ \V \\ p \mat{I}_c  \\  q \mat{I}_c 
\end{pmatrix},$$
so that equation~(\ref{eq:lateq}) becomes
\begin{equation}
\label{eq:hAx}
 \vec{h} = \vec{x} \mat{A}
\end{equation}
Furthermore we define the  matrix $\mat{B}$ of size 
$2(c+\alpha) \times 2(c+\alpha) $ as the block matrix
\begin{equation*}
  \vec{B} = \begin{pmatrix}
    \mat{I}_\alpha/p & \mat{0}_{\alpha\times\alpha} & \V/p &
\mat{0}_{\alpha\times\alpha} \\
    \mat{0}_{\alpha\times\alpha} & \mat{I}_\alpha/q & \mat{0}_{\alpha\times c}
& \V/q \\
    \mat{0}_{c\times\alpha} & \mat{0}_{c\times\alpha} & \mat{I}_c
& \mat{0}_{c\times c} \\
    \mat{0}_{c\times\alpha} & \mat{0}_{c\times\alpha} & \mat{0}_{c\times c}
& \mat{I}_c
\end{pmatrix}
\end{equation*}
and the vector $\vec{u}$ of length $2(c+\alpha)$ as
$$ \vec{u} = (
\underbrace{0,\ldots,0}_{2\alpha},
\frac{p-1}{2p} \vec{e}_c, 
\frac{q-1}{2q} \vec{e}_c) .$$
Now the inequalities  \eqref{eq:latnorm1} and 
$\|\vec{r}\|_{\infty}< p/2$,
$\|\vec{t}\|_{\infty}< q/2$ 
 are equivalent to the single inequality
\begin{equation}
\label{eq:infnorm}
\left|\left| \vec{x} \mat{B} - \vec{u} \right|\right|_\infty
< \frac{1}{2} .
\end{equation}

So finding a solution to the MMO problem is equivalent to finding
an integer solution of equation~\eqref{eq:hAx} that satisfies
the constraint from  inequality~\eqref{eq:infnorm}.

Let $\vec{x}_0$ be an arbitrary integer
solution to equation~\eqref{eq:hAx}, for example we can take 
$\vec{x}_0= (\underbrace{0,\ldots,0}_{2\alpha},
\mu_1 \vec{h}, \mu_2 \vec{h})$. 
Every integer solution $\vec{x}$
of equation~\eqref{eq:hAx} can now be written as $\vec{x}=\vec{x}_0 +
\vec{y}$, where $\vec{y}\mat{A}=0$. 
Thus $\vec{y}$ lies in the left integer kernel of $\mat{A}$, which is spanned
by the rows of the matrix
\[  \mat{K} = \begin{pmatrix}
    \mat{I}_\alpha & -\mat{I}_\alpha & \mat{0}_{\alpha\times c}
& \mat{0}_{\alpha\times c} \\
\mat{0}_{\alpha\times\alpha} & \mat{I}_\alpha & -\mu_1\V & -\mu_2 \V \\
\mat{0}_{c\times\alpha} & \mat{0}_{c\times\alpha} & q \mat{I}_c &
-p \mat{I}_c \end{pmatrix},\]
so $\vec{y}=\vec{w}\mat{K}$ with $\vec{w}\in\Z^{2\alpha+c}$.
Substituting this into
equation~\eqref{eq:infnorm}, we obtain
\begin{equation*}
  \left|\left| \vec{w}\mat{K}\mat{B} - (\vec{u}-\vec{x}_0\vec{B})
    \right|\right|_\infty < \frac{1}{2}.
\end{equation*}
In other words, we are looking for vectors in the lattice $\cL$
 spanned by the rows of the matrix
\[ \mat{C} = \mat{K}\mat{B} = \begin{pmatrix}
\mat{I}_\alpha/p & -\mat{I}_\alpha/q & \V/p & -\V/q \\
\mat{0}_{\alpha\times\alpha} & \mat{I}_\alpha/q & -\mu_1\V &
\mu_1 p\V/q \\
\mat{0}_{c\times\alpha} & \mat{0}_{c\times\alpha} & q\mat{I}_c & 
-p\mat{I}_c\end{pmatrix} \]
that have  distance less than $1/2$ in infinity norm to the vector
$\vec{u}-\vec{x}_0\mat{B}$.

The main idea of the lattice reduction technique is to show that the close
vector is unique. Suppose we have two lattice vectors
$\vec{z}_1$ and $\vec{z}_2 \in\cL$  satisfying 
\begin{equation*}
  \left|\left| \vec{z}_i - (\vec{u}-\vec{x}_0\vec{B})
    \right|\right|_\infty < \frac{1}{2},\quad  i=1,2,
\end{equation*}
then $\vec{z}=
\vec{z}_1- \vec{z}_2 \in {\cL}$ and $ \left\|\vec{z}\right\|_\infty < 1$. 

Note that the fourth block column of $\mat{C}$ is equal to $-p/q$ times the
third. This implies that for each $\vec{z}\in\cL$, we have that
$\|\vec{z}\|_\infty = \|\vec{z}'\|_\infty$, where $\vec{z}'\in\Q^{2\alpha+c}$
is obtained from $\vec{z}$ by deleting the last block of $c$ coordinates
if $p<q$ and the third block if $q<p$.
Deleting the corresponding block column from $\mat{C}$ gives a square matrix
$\mat{C}'$; the $(2\alpha+c)$-dimensional
lattice of which the rows of $\mat{C}'$ are a basis is
denoted $\cL'$. Then
\[ \vol(\cL') = | \det(\mat{C}')| = \frac{ \max(p,q)^c }{(pq)^\alpha}. \]

The Gaussian heuristic suggests that a $d$-dimensional
lattice $\cL'$ with volume $\vol(\cL')$ is unlikely to have a nonzero
vector which is substantially shorter (in infinity norm)
 than $(1/2) \vol(\cL')^{1/d}$.
Thus, if $\vol(\cL')> 2^{2\alpha+c}$
it is likely that the close vector is
unique. When $p$ and $q$ have similar magnitude, we therefore conclude that 
if $c$ is somewhat larger than $2\alpha$, it is likely that the MMO problem can
be solved.

Conversely, with elementary methods we can show that if $p$ and $q$ have similar magnitude, then reconstruction of $(f,g)$ requires that 
on average, $c$ is at least $2\alpha$. Indeed, the number of pairs of polynomials $(f,g)$ equals
$(pq)^{\alpha}$; the number of sequences of  function values in $c$ integers equals $(p+q-1)^c$. Hence, if $(pq)^{\alpha} > 
(p+q-1)^c$, then there exists a sequence of function values that can be generated by more than one pair $(f,g)$ of polynomials.
The following proposition gives a slightly stronger result.
 \begin{propo}\label{ludo_bounds} If $p$ and $q$ have similar magnitude, then on average the minimum  
number of required  values  to compute the polynomials  $\modulo{f(X)}{p}$ and $\modulo{g(X)}{q}$
 is at least $2\alpha.$
 \end{propo}
 \begin{proof}
Let $x_1,\ldots, x_c$ be integers.
For ${\bf y}\in Y=\{0,1,\ldots, p+q-2\}^c$, we define
\[ N({\bf y}) = \;\;  \mid \{ (f,g) \in \Z_p[x]\times \mathbb{Z}_q[x] \mid
  \mbox{deg}(f)\leq \alpha, \mbox{deg}(g)\leq \alpha, f(0)=g(0)=0 \mbox{ and } \]
\[ \langle f(x_i)\rangle_p + \langle g(x_i)\rangle_q = y_i \mbox{ for } 1\leq i\leq c\}\mid . \]
Of course, we have that 
\[ \sum_{\bf y\in Y} N({\bf y}) = p^{\alpha}q^{\alpha} . \]
We assume the polynomials $f$ and $g$ are chosen uniformly and independently.
Then the probability $p({\bf y})$ to observe ${\bf y}\in Y$ equals $N({\bf y})/(pq)^{\alpha}$. 
The expected number ${\cal E}$ of pairs of polynomials ($f,g$) matching ${\bf y}\in Y$ thus satisfies
\[ {\cal E} = \sum_{{\bf y}\in Y} N({\bf y})p({\bf y}) = \frac{1}{(pq)^{\alpha}} \sum_{{\bf y}\in Y} 
N({\bf y})^2 \geq
\frac{1}{(pq)^{\alpha}} \frac{\left(\sum_{{\bf y}\in Y}N({\bf y})\right)^2}{|Y|} = 
\frac{(pq)^{\alpha}}{(p+q-1)^c}, \]
where the inequality sign follows from the Cauchy-Schwarz inequality. \\
Consequently, if $c\leq 2\alpha-1$, then ${\cal E}\geq \frac{(pq)^{\alpha}}{(p+q-1)^{2\alpha-1}} \geq
(p+q)\left(\frac{pq}{(p+q)^2}\right)^{\alpha}$. And so, writing $q=p(1+\epsilon)$, we have that
${\cal E}\geq p(2+\epsilon)\left( \frac{1+\epsilon}{(2+\epsilon)^2}\right)^{\alpha}$.  

For sufficiently small $\epsilon$, we thus have that ${\cal E}>1$.  
 \end{proof}

In the next section we provide the details of the resulting algorithm and the performance of our Sage implementation.

\section{ The algorithm and its implementation}

The basic structure of the algorithm is the following:

\begin{algorithm}
  \caption{Algorithm to solve MMO problem}
  \begin{algorithmic}
    \REQUIRE Set $J$ and $p, q$
    \ENSURE  $ \langle f(X) \rangle_p$ and $ \langle g(X) \rangle_q$.
    \STATE Generate vectors, $\vec{h},\vec{x}_0,\vec{u}$ and matrices $\vec{B},\vec{K},\vec{C}$
    as defined in Section 2.
    \STATE Use a Closest Vector algorithm to find $\vec{x}'$.
    \RETURN the polynomials with coefficients equal to the first $2\alpha$
    components of vector $\vec{x}'$
  \end{algorithmic}
\end{algorithm}

This is the pseudocode of the algorithm we have used to compute a
close vector which is called the  \textit{Babai Nearest Plane Algorithm}, see~\cite{Babai86}: 

\begin{algorithm}
    \caption{Babai Nearest Plane algorithm}
  \begin{algorithmic}
    \REQUIRE Basis given as a matrix $\vec{B}$, $\vec{t}$ 
    \ENSURE A
    vector $\vec{u}\in\cL(\vec{B})$, such that
    $\|\vec{u}-\vec{t}\|_2\le 2^{d/2}\min\{\|\vec{v}-\vec{t}\|_2\;|\;
    \vec{v}\in\cL(B)\}$ 
    \STATE Run LLL algorithm on matrix $\vec{B}$
    with standard $\epsilon = 3/4$ 
    \STATE $\vec{b} = \vec{t}$ 
    \FOR{j from $n$ to $1$} 
    \STATE $c_j = \lceil\frac{\vec{b}\overline{
        \vec{b}_j}}{\|\overline{\vec{b}_j}\|_2^2} \rceil$ 
    \STATE $\vec{b} = \vec{b} - c_j \vec{b}_{j}$
    \ENDFOR
    \RETURN $\vec{b}-\vec{t}$
  \end{algorithmic}
\end{algorithm}
We have implemented our algorithm for solving the MMO problem in the
Sage system, including the Babai algorithm.  

Babai Nearest Plane algorithm finds a close vector with
respect to the Euclidean norm.
The closest vector with respect to the
infinity norm can be found doing the following computations:
\begin{itemize}
\item Calculate a $LLL$-reduced basis $\vec{a}_1,\ldots, \vec{a}_d$.
\item Calculate a close vector $\vec{b}$ using the Babai Nearest Plane algorithm.
\item Take the  vector $\vec{b}'$  that minimizes
  $\|\vec{t}-\vec{b}'\|_{\infty}$ where $\vec{b}'$ belongs to the
  following set,
  \begin{equation*}
    \{\vec{b}'\;|\; \vec{b}' = \vec{b}+\sum_{i=1}^{d}C_i\vec{a}_i,\;
     |C_i| \le  \sqrt{d} 2^{(i-1)/2},\;
    i = 1,\ldots, d
    \}.
  \end{equation*}
\end{itemize}
The fact that this returns the closest vector with respect to the
infinity norm comes from~\cite[Proposition 1.6]{LLL82} and  the proof
of~\cite[Proposition 1.11]{LLL82}. 

To test when
the algorithm to solve MMO works, we use an indirect 
method. We take the lattice defined by the rows of  $\vec{C}$
and check for the shortest vector. If this vector has norm bigger than
$1$, then we know that the algorithm will work and in other case, we
suppose that it fails.

In this way, we will count as fails many cases where the algorithm
could possibly work. However, implementations show that, even in these
conditions, the algorithm for solving MMO seems to work in most of the
cases.  To be more precise, selecting uniformly at random  
$c=2\alpha$ values  $x_i \in  [1, p]$ the algorithm was successful 
in $100\%$ of the cases with $200$-bit number $p$. This 
confirms that $c=2\alpha$ is indeed the natural threshold 
for the algorithm.

However the  perfomance changes  if  the values are selected from a
small  interval  $[1, p^{1/K}]$ for big $K$. If $K$ is smaller than
$\alpha$ then it is possible to recover some of the coefficients of
the polynomials. More precisely, the algorithm recovers the
coefficients of the polynomials of the monomials of degree greater
than $K$.

This fact is interesting because of the design of the HIMMO  key
generation system~\cite{MorchonTGG12} and it is analyzed in detail in
next section.

\section{Restriction to small arguments}

In Proposition~\ref{propo:unique} we showed that $f$ and $g$ are determined up
to a constant if $h(x)=\langle f(x)\rangle_p + \langle g(x)\rangle_q$ for all
$x\in\Z$. This constant can be fixed by setting $f(0)=g(0)=0$.
However, in cryptographic applications, values of $x$
that can be used are from a smaller interval: $0\leq x < w$, where
$w\approx(\min(p,q))^{1/K}$ for some $K\geq 1$. If we are interested only in 
the function $h$ on this short interval, then the reconstruction is typically
far from
unique. In fact, let $C\in\Q[x]$ be a polynomial of degree at
most $K$ that takes integer values for all integer arguments, i.e.,
$C$ is an integer linear combination of binomial coefficients:
\[ C(x) = \sum_{k=0}^K C_k \binom{x}{k},\quad C_0,\ldots,C_K\in\Z.\]
If $\gcd(p,K!)=\gcd(q,K!)=1$, the factorials $2!, 3!,\ldots,K!$ have inverses modulo $p$ and modulo $q$, so we can define polynomials
$c_p\in\Z_p[x]$ and $c_q\in\Z_q[x]$ of degree at most $K$,
such that for all integer $x$:
\[ \langle c_p(x)\rangle_p = \langle C(x)\rangle_p \text{ and }
\langle c_q(x) \rangle_q = \langle C(x) \rangle_q. \]

If it holds that $C$ is small on $[0,w)$, in the sense that
\[ 0 \leq \langle f(x)\rangle_p + C(x) \leq p-1 \text{ and }
0 \leq \langle g(x)\rangle_q - C(x) \leq q-1 \text { for all integer }
x\in[0,w),\]
then $f+c_p$ and $g-c_q$ decompose $h$ on $[0,w)$.

If all short lattice vectors correspond to such pairs $(c_p,-c_q)$, then all
lattice points close to our target vector correspond to polynomials
$(\tilde{f},\tilde{g})$ that also decompose $h$. In other words: though
we cannot reconstruct $f$ and $g$, we can interpolate $h$ correctly.

Note that our previous analysis based on lattice volumes and the Gaussian
heuristic failed to see the short vectors that correspond to the polynomials
$C(x)$. This should not be surprising: the lattice volume is independent of the
values $x_1,\ldots,x_c$, and these short vectors appear only if $0\leq x_i <w$
for $i=1,2,\ldots,c$.
The numerical experiments show that the Gaussian heuristic is not valid for $\cL'$ when the $x_i$ are
from an interval that is much shorter than $p$ and $q$.

Above, we found a sublattice of $\cL'$ with short basis vectors. One may
wonder if there are short vectors in $\cL'$ that are not in the sublattice
generated by these short vectors.
To answer this question, we apply the Gaussian heuristic to the lattice that is
obtained when the sublattice is projected out as in Section~6.1 of
\cite{GarciaRietmanShparlinskiTolhuizenpreprint2014}.
Lemma~5 on page 29 of~\cite{BookLLL} gives the explicit formula for the volume
of a lattice resulting as the orthogonal projection over a linear subspace.  We
write it here for the convenience  of the reader.
\begin{Lm}
  Let $L$ be a $d$-dimensional lattice in $\R^s$ and $M$ be a $r$-dimensional sublattice of L which the property that 
  one of its basis can be extended to a basis of $L$.
  Let $\pi_M$ denote the orthogonal projection over the orthogonal suplement of the linear span of $M$.
  Then the image of $L$ by $\pi_M$ is a $(d-r)$-dimensional lattice of $\R^s$ and volume 
  $\vol{(L)}/\vol{(M)}$.
\end{Lm}

Assuming $p<q$ and $K>\alpha$, the volume of the resulting lattice equals
\[ q^{c-\alpha-1}/\sqrt{\det(B B^t)},\]
where $B$ is the matrix
\[ B  =  \begin{pmatrix}
1 & 1 & \cdots & 1 \\
x_1 & x_2 & \cdots & x_c \\
\binom{x_1}{2} & \binom{x_2}{2} & \cdots &\binom{x_c}{2} \\
\vdots & \vdots & & \vdots \\
\binom{x_1}{\alpha} & \binom{x_2}{\alpha} & \cdots & \binom{x_c}{\alpha} 
\end{pmatrix}.\]
When $x_1,\ldots,x_c$ are uniformly drawn from $[0,w)$, $\sqrt{\det(B B^t)}$
will be of order $w^{\alpha(\alpha+1)/2}$. Comparing powers of $w$, the
resulting volume is therefore expected to be much larger than~1 if
$ K (c-\alpha) > \alpha(\alpha+1)/2.$

\textbf{Example}.  Let $\alpha=6$, $w=2^{16}$, and
\begin{align*}
p&=322503631145131659181549502994177879533 \\
q&=322503631145131659181549502996361408083
\end{align*}
so that $p\approx q\approx w^8$, and $K=8$.

Suppose the coefficients of $f\in\Z_p[x]$ and $g\in\Z_q[x]$ are equal to
\begin{align*}
f_0 &= 192299855391930388766069561100536978455 \\
f_1 &= 80324299466086676640269450973128212279 \\
f_2 &= 134802655995538131612821059755185358806 \\
f_3 &= 223036273860653058471857675170774765711 \\
f_4 &= 81615146624468266057406642183853219751 \\
f_5 &= 282812473825451509017913772106035640705 \\
f_6 &= 278906905917307720980382059680001096297
\end{align*}
and
\begin{align*}
g_0&= 81564018199971421800339434244552477506  \\
g_1 &= 12324696623153181384549093381069011068  \\
g_2 &= 80030936209387920933656861269029654371 \\ 
g_3 &=315635911037272927490950126509525457405 \\
g_4 &= 217950416300798270685940703747161570332 \\
g_5 &= 75454198535432609870859677101539890163 \\ 
g_6 &= 26892964982895845277700750286746366172 .
\end{align*}
In this example the smallest value of $c$ for which $K (c-\alpha-1) >
\alpha(\alpha+1)/2$ is $10$. That means that if we pick $c=10$ points uniformly
from~$[0,w)$, there is a fair chance that the volume of the projected lattice is much larger than~1.

We are given the values of $h(x)=\langle f(x)\rangle_p + \langle g(x)\rangle_q$
in $c=10$ points randomly chosen form the interval $[0,w)$
according to the following table.\\
\begin{center}
\begin{tabular}{|r|r|r|}
$i$ & $x_i$ & $h(x_i)$ \\
\hline
1 & 34915 & 357083778061836956769804023406098677550 \\
2 & 30844 & 501434122478371565756095361502998185705 \\
3 & 55453 & 362669734592545590446623074678041228580 \\
4 & 43386 & 453528102619044436291771088280150310990 \\
5 & 61725 & 409617140945520234057946967178875528708 \\
6 & 39144 & 426802401636630448727954157743588116409 \\
7 & 14608 & 311556461063783252602939114845129657070 \\
8 & 24287 & 594980681560119885662989234834546277705 \\
9 & 24582 & 119430230752341918846040173886171897211 \\
10& 36432 & 20159634491993343981036574887019110187
\end{tabular}
\end{center}

Constructing the lattice as described before, including an additional row of
ones in the matrix $\V$ in order to take the constant terms of the polynomials
into account, we find a lattice vector that is close to the target vector. The
polynomial coefficients corresponding to this nearby lattice vector are
\begin{align*}
\tilde{f}_0 &= 136931826884319377850275846232659046764 \\
\tilde{f}_1 &= 127274522470810992144873423947517028220 \\
\tilde{f}_2 &= 166540029496138250784732903087691991725 \\
\tilde{f}_3 &= 149982375974823828230059543913714128152 \\
\tilde{f}_4 &= 157228597180650695773338976918720767558 \\
\tilde{f}_5 &= 159036774649843108350794952315211705687 \\
\tilde{f}_6 &= 151078581747150708184414679388065540292 
\end{align*}
and
\begin{align*}
\tilde{g}_0 &= 136932046707582432716133148636421185297 \\
\tilde{g}_1 &= 126626289190994695470719871904637347436 \\
\tilde{g}_2 &= 155794773090498354822261518935095270543 \\
\tilde{g}_3 &= 169208171060443111900860401561223641003 \\
\tilde{g}_4 &= 146816182843853780680863223220064874839 \\
\tilde{g}_5 &= 149958509619423673718575100602091084672 \\
\tilde{g}_6 &= 150242072053814918362813276371264593533 .
\end{align*}

The difference $\tilde{h}(x)-h(x)$ is plotted in Figure~\ref{fig:interpolation}.
This shows that, even though the interpolation is not perfect, it still
predicts the correct value in a sizable fraction of the points and the error
pattern does not look random.

\begin{figure}[H]
\includegraphics[width=.9\textwidth]{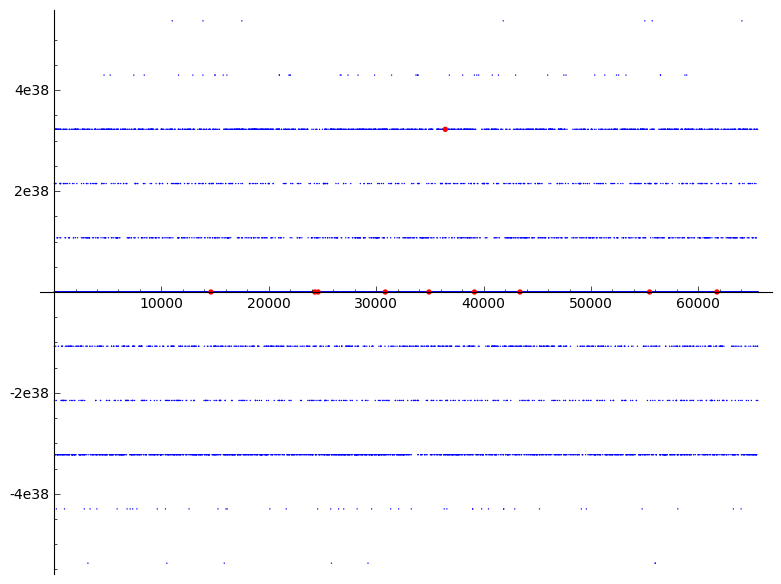}
\caption{Graph of $\tilde{h}(x)-h(x)$. The reconstucted function $\tilde{h}(x)$
fits the observation perfectly in  9 out of the 10 points, but more
interestingly, the interpolation error is zero in many other points, even
though $c < 2\alpha$. If the error is non-zero, it is restricted to very narrow
bands.}
\label{fig:interpolation}
\end{figure}

\section{Conclusions}
We have introduced the MMO problem. It seems infeasible to solve the MMO problem for unknown moduli. We have shown the equivalence of the MMO problem to finding close vectors in a lattice. If all observed function arguments lie in an interval that is much shorter than the moduli, then reconstruction of the unknown polynomials is infeasible; however, the computed polynomials often gives correct interpolation of the function on that short interval.

The MMO problem can readily be generalized to more than two moduli. Furthermore, an additional modular operation may be performed on the sum of the polynomial evaluations, see \cite{MorchonTGG12}.

\bibliographystyle{IEEEtran}

\bibliography{references}

\begin{thebibliography}{10}
\providecommand{\url}[1]{#1}
\csname url@samestyle\endcsname
\providecommand{\newblock}{\relax}
\providecommand{\bibinfo}[2]{#2}
\providecommand{\BIBentrySTDinterwordspacing}{\spaceskip=0pt\relax}
\providecommand{\BIBentryALTinterwordstretchfactor}{4}
\providecommand{\BIBentryALTinterwordspacing}{\spaceskip=\fontdimen2\font plus
\BIBentryALTinterwordstretchfactor\fontdimen3\font minus
  \fontdimen4\font\relax}
\providecommand{\BIBforeignlanguage}[2]{{%
\expandafter\ifx\csname l@#1\endcsname\relax
\typeout{** WARNING: IEEEtran.bst: No hyphenation pattern has been}%
\typeout{** loaded for the language `#1'. Using the pattern for}%
\typeout{** the default language instead.}%
\else
\language=\csname l@#1\endcsname
\fi
#2}}
\providecommand{\BIBdecl}{\relax}
\BIBdecl

\bibitem{Boyar}
J.~Boyar, ``Inferring sequences produced by pseudo-random number generators,''
  \emph{Journal of Association of Computing Machinery}, vol.~36, no.~1, pp.
  129--141, 1989.

\bibitem{MorchonTGG12}
O.~Garc\'{\i}a-Morchon, L.~Tolhuizen, D.~G{\'o}mez, and J.~Gutierrez, ``Towards
  fully collusion-resistant id-based establishment of pairwise keys,''
  \emph{IACR Cryptology ePrint Archive}, vol. 2012, p. 618, 2012.

\bibitem{micciancio-goldwasser}
D.~Micciancio and S.~Goldwasser, \emph{Complexity of lattice problems}, ser.
  The Kluwer International Series in Engineering and Computer Science,
  671.\hskip 1em plus 0.5em minus 0.4em\relax Boston, MA: Kluwer Academic
  Publishers, 2002, a cryptographic perspective.

\bibitem{Gruber}
\BIBentryALTinterwordspacing
P.~M. Gruber and C.~G. Lekkerkerker, Eds., \emph{Geometry of numbers}, ser. The
  Kluwer International Series in Engineering and Computer Science.\hskip 1em
  plus 0.5em minus 0.4em\relax Amsterdam: North-Holland Publishing Co., 1987,
  vol.~37. [Online]. Available:
  \url{http://dx.doi.org/10.1007/978-3-642-02295-1}
\BIBentrySTDinterwordspacing

\bibitem{GrLoSch}
M.~Gr\"otschel, L.~Lov\'asz, and A.~Schrijver, \emph{Geometric algorithms and
  combinatorial optimization}.\hskip 1em plus 0.5em minus 0.4em\relax Berlin:
  Springer-Verlag, 1993.

\bibitem{BookLLL}
\BIBentryALTinterwordspacing
P.~Q. Nguyen and B.~Vall{\'e}e, Eds., \emph{The {LLL} algorithm}, ser.
  Information Security and Cryptography.\hskip 1em plus 0.5em minus 0.4em\relax
  Berlin: Springer-Verlag, 2010, survey and applications. [Online]. Available:
  \url{http://dx.doi.org/10.1007/978-3-642-02295-1}
\BIBentrySTDinterwordspacing

\bibitem{LLL82}
\BIBentryALTinterwordspacing
A.~K. Lenstra, H.~W. Lenstra, Jr., and L.~Lov{\'a}sz, ``Factoring polynomials
  with rational coefficients,'' \emph{Math. Ann.}, vol. 261, no.~4, pp.
  515--534, 1982. [Online]. Available:
  \url{http://dx.doi.org/10.1007/BF01457454}
\BIBentrySTDinterwordspacing

\bibitem{Igor-Arne}
\BIBentryALTinterwordspacing
I.~Shparlinski and A.~Winterhof, ``Noisy interpolation of sparse polynomials in
  finite fields,'' \emph{Appl. Algebra Eng., Commun. Comput.}, vol.~16, no.~5,
  pp. 307--317, Nov. 2005. [Online]. Available:
  \url{http://dx.doi.org/10.1007/s00200-005-0180-1}
\BIBentrySTDinterwordspacing

\bibitem{Babai86}
L.~Babai, ``On lov{\'a}sz' lattice reduction and the nearest lattice point
  problem,'' \emph{Combinatorica}, vol.~6, no.~1, pp. 1--13, 1986.

\bibitem{GarciaRietmanShparlinskiTolhuizenpreprint2014}
\BIBentryALTinterwordspacing
O.~Garc\'{\i}a-Morchon, R.~Rietman, I.~E. Shparlinski, and L.~Tolhuizen,
  ``Interpolation and approximation of polynomials in finite fields over a
  short interval from noisy values,'' \emph{arXiv.org preprint archive}, vol.
  2014, 2014. [Online]. Available: \url{http://arxiv.org/abs/1401.1331}
\BIBentrySTDinterwordspacing

\end{thebibliography}

\end{document}